\documentclass[10pt,twoside]{article}
\usepackage{amsmath,amsthm}

\date{}
\begin{document}

\centerline{}

\centerline{\Large{\bf Extension of Fuzzy h-ideals in $\Gamma$-hemirings}}

\centerline{}

\centerline{\bf{Sujit Kumar Sardar$^{1}$ and Debabrata Mandal$^{2}$}}

%\centerline{\bf{Sujit Kumar Sardar and Debabrata Mandal}}

\centerline{Department of Mathematics}

\centerline{Jadavpur University, Kolkata}

\centerline{E-mail: sksardarjumath@gmail.com$^{1}$, dmandaljumath@gmail.com$^{2}$}

%\centerline{E-mail: dmandaljumath@gmail.com}

\centerline{}

\newtheorem{Theorem}{\quad Theorem}[section]

\newtheorem{Definition}[Theorem]{\quad Definition}

\newtheorem{Corollary}[Theorem]{\quad Corollary}

\newtheorem{Lemma}[Theorem]{\quad Lemma}

\newtheorem{Remark}[Theorem]{\quad Remark}

\newtheorem{Example}[Theorem]{\quad Example}

\newtheorem{Proposition}[Theorem]{Proposition}
\def\proofname{\indent {\sl\textbf{Proof.} }}

\begin{abstract} In this paper the concept of the extension of fuzzy h-ideals in one-sided $\Gamma$-hemiring is introduced and
some of its properties are investigated. Specially we have studied the extension of prime fuzzy h-ideals in $\Gamma$-hemirings and gave its characterization.
\end{abstract}
\textbf{2010 AMS Classification:} 03E72


\section{Introduction}
There are many concepts of universal algebras generalizing an
associative ring $(R,+,\cdot )$.
 Some of them have been found to be very useful for solving problems
 in different areas of applied mathematics and information sciences. Semiring is one such concept. In 1934 Vandiver \cite{re:van} introduced this concept.
 The concept of a semiring generalizes that of a
ring, allowing the additive structure to be only a semigroup
instead of a group.
The set of nonnegative integers $Z_{0}^{+}$
with usual addition and multiplication provides a natural example
of a semiring.\\
But the situations  for the set of all nonpositive integers
$Z_{0}^{-}$ is
different. They do not form semirings with the above operations
because multiplication in the above sense is no longer a binary
compositions.\\
To solve this problem, in 1995 M.M.K. Rao\cite{Rao} introduced a new kind of algebraic
structure what is known as $\Gamma$-semiring which is a generalization of  semiring as well as of  $\Gamma-$ring.
To make $\Gamma$-semiring more effective,
 Dutta and Sardar introduced the
notion of operator semirings  \cite{re:Dutta} of a
$\Gamma$-semiring. By
using the interplay between the operator semirings and the
$\Gamma$-semiring they enriched the theory of $\Gamma$-semirings.\\
Ideals of semiring play a central role in the structure theory and
   useful for many purposes. However they do not in general coincide
   with the usual ring ideals(e.g. an ideal
of a semiring need not be the kernel of a semiring morphism) and for this reason, their use is somewhat
    limited in trying to obtain analogues of ring theorems for semiring.
    To solve this problem,
 Henriksen\cite{Henriksen}, Iizuka\cite{Iizuka} and
LaTorre\cite{LaT} investigated a more restricted class of ideals which are called h-ideals and k-ideals. Jun et al\cite{Ybjun} and Zhan et al\cite{Zhan} studied these ideals in
hemirings in terms of fuzzy subsets. We (\cite{SD1},\cite{SD2}) extended these concepts to the theory of
$\Gamma$-hemirings.\\
Motivated by Xie\cite{Xie}, as a continuation of this we study the concept of the extension of
fuzzy h-ideals in $\Gamma$-hemirings in \cite{SD8}. In that paper, we take the $\Gamma$-hemiring to be both-sided i.e. when $S$ is a $\Gamma$-hemiring, $\Gamma$ is also a $S$-hemiring.  Here we take the $\Gamma$-hemiring to be one-sided and study all those properties of \cite{SD8}.
\section{Preliminaries}

\begin{Definition} Let S and $\Gamma$ be two additive
commutative semigroups with zero. Then S is called a
$\Gamma$-hemiring if there exists a mapping $\\ S \times \Gamma
\times S \rightarrow S $ ( (a,$\alpha,$b) $\mapsto$ a$\alpha$b) satisfying the following conditions:\\
(i) $(a+b) \alpha c =a \alpha c+ b \alpha c$,\\
(ii) $ a \alpha (b+c)= a \alpha b+a \alpha c $,\\
(iii) $a (\alpha+\beta)b=a \alpha b +a \beta b$,\\
(iv) $ a \alpha (b \beta c)=(a \alpha b) \beta c $.\\
(v) $ 0_{S} \alpha a=0_{S} =a \alpha 0_{S},$\\
(vi)$ a0_{\Gamma}b=0_{S}=b 0_{\Gamma} a$\\
for all $a,b,c \in S $ and for all $\alpha, \beta \in \Gamma $.
\\
For simplification we write 0 instead of $0_{S}$ and $0_{\Gamma}$.\\
A $\Gamma$-hemiring S is called commutative if $a\alpha b=b\alpha a $ for all a,b$\in$S and $\alpha\in\Gamma$.
\end{Definition}
 Throughout this paper S denotes a $\Gamma$-hemiring with zero.

\begin{Definition} A left ideal A of a $\Gamma$-hemiring S is
called a left h-ideal if for any x, z $\in$ S and a, b $\in$ A,
x + a + z = b + z $\Rightarrow$ x $\in$ A.\\
A right h-ideal is defined analogously.
\end{Definition}
\begin{Definition} Let S be a $ \Gamma
$-hemiring. A proper h-ideal I of S is said to be prime if for any two
h-ideals H and K of S, $H\Gamma K \subseteq I$ implies that either $H
\subseteq I$ or $K \subseteq I$.
\end{Definition}

\begin{Theorem}\label{Th:2.3} \textnormal{\cite{SD2}}If I is an h-ideal of a $\Gamma$-hemiring S then the
following conditions are equivalent:\\
(i) I is a prime h-ideal of S.\\
(ii) If $a\Gamma S \Gamma b \subseteq I$ then either a $\in$ I or b
$\in$ I where a,b $\in$ S.
\end{Theorem}
\begin{Definition}\cite{SD5} Let $\mu $ and $\theta$ be two fuzzy sets of
a $\Gamma-$hemiring S. Define a generalized h-product of $\mu $ and $\theta$ by
$$
\begin{array}{ll}
\mu
o_{h}\theta(x)&=\underset{x+\displaystyle{\sum_{i=1}^{n}a_{i}\gamma_{i}b_{i}+z=
\sum_{i=1}^{n}c_{i}\delta_{i}d_{i}+z}}{\sup[\underset{i}{\min}\{\min
\{\mu(a_{i}),\mu(c_{i}),\theta(b_{i}), }
\theta(d_{i})\}\}]\\
&= 0, \textmd{if x cannot be expressed as above}
\end{array}
$$
where x,z,$a_{i},b_{i},c_{i},d_{i}\in $S and
$\gamma_{i},\delta_{i}\in \Gamma,$ for i=1,...,n.
\end{Definition}

\begin{Definition} Let $\mu$ be the non empty fuzzy subset of
 a $\Gamma$-hemiring S (i.e. $\mu(x) \neq 0 $ for some $x \in S
 $). Then $\mu $ is called a fuzzy left h-ideal [ fuzzy right h-ideal]
 of S if \\(i) $\mu(x+y) \geq \min \{ \mu(x), \mu(y)\}$ and
  \\~~~~~~~~(ii) $\mu(x \gamma y) \geq \mu(y) $ [resp. $\mu(x \gamma y) \geq
  \mu(x)]$ for all $ x,y \in S, \gamma \in \Gamma$.
    \\~~~~~~~~(iii) For all x,a,b,z$\in$S, x+a+z=b+z implies $\mu(x)\geq\min\{\mu(a),\mu(b)\}$.
    \\
A fuzzy ideal of a $\Gamma$-hemiring S is a non empty fuzzy subset
of S which is a fuzzy left ideal as well as a fuzzy right ideal of
S.\\
A fuzzy h-ideal $\mu$ of S is called a fuzzy h-bi-ideal, fuzzy h-interior ideal if for all x,y,z,a,b$\in$S and
 $\alpha,\beta\in\Gamma$, $\mu(x\alpha y\beta z)\geq\min\{\mu(x),\mu(z)\}$, $\mu(x\alpha y \beta z )\geq\mu(y)$,
 respectively.\\
 A fuzzy subset $\mu$ of a $\Gamma$-hemiring S is
 called fuzzy h-quasi-ideal if $\mu$ satisfies (i) and (iii) along with the condition $(\mu o_{h}\chi_{S})\cap(\chi_{S}
 o_{h}\mu)\subseteq\mu$, where $\chi_{S}$ is the characteristic function of S.
 \end{Definition}

Now we recall following definitions and result from \cite{Xma2} for subsequent use.
\begin{Definition} Let $\mu $ and $\theta$ be two fuzzy sets of
a $\Gamma-$hemiring S. Define h-product of $\mu $ and $\theta$ by\\
$\mu\Gamma_{h}\theta(x)=\underset{x+a_{1}\gamma b_{1}+z=a_{2}\delta
b_{2}+z}{\sup[\min\{\mu(a_{1}),\mu(a_{2}), }
\theta(b_{1}),\theta(b_{2})\}] \\
~~~~~~~~~~=0,\textmd{if x cannot be expressed as } x+a_{1}\gamma
b_{1}+z=a_{2}\delta b_{2}+z $\\
for x,z,$a_{1},a_{2},b_{1},b_{2}\in$ S and $\gamma, \delta \in
\Gamma.$
\end{Definition}

\begin{Definition} A fuzzy h-ideal $\mu$ of a $\Gamma$-hemiring S
is said to be prime(semiprime) if $\mu$ is not a constant function and
for any two fuzzy h-ideals $\sigma $ and $\theta$ of S, $\sigma \Gamma_{h}
\theta \subseteq \mu$ implies that either $\sigma \subseteq \mu$ or
$\theta \subseteq \mu$(resp. $\theta \Gamma_{h} \theta \subseteq \mu$
implies $\theta \subseteq \mu$).
\end{Definition}

\begin{Theorem}\label{Charc. Prime}Let $\mu$ be a fuzzy h-ideal of S. Then $\mu$ is a
prime fuzzy h-ideal of S if and only if the following conditions hold

(i) $\mu(0)=1$,

(ii) $ Im ~\mu =\{1, t\}, ~~t \in [0,1)$,

(iii) $\mu_{0}=\{x \in S: \mu(x)=\mu(0)\}$ is a prime h-ideal of S.
 \end{Theorem}

\section{Fuzzy h-ideal extension in $\Gamma$-hemirings}

\begin{Definition} Let $\mu$ be a fuzzy subset of S and $x \in S$. Then the fuzzy subset $<x,\mu>$ of S, defined by
$$<x,\mu>(y)=\underset{s\in S}{\underset{\alpha,\gamma\in\Gamma} {\inf}}\mu(x \alpha s \gamma y)$$ for all $y \in S$, is
called the extension of $\mu$ by x.
\end{Definition}

\begin{Theorem}\label{Prop 3.2} Let $\mu$ is a fuzzy right h-ideal of S and $x \in S$. Then the extension $<x,
\mu>$ is a fuzzy right h-ideal of S.
\end{Theorem}

\begin{proof}
Let p,q,a,b,z$\in$S and $\beta\in\Gamma$. Then\\
$$
\begin{array}{ll}
<x,\mu>(p+q) &=\underset{s\in S}{\underset{\alpha,\gamma\in\Gamma} {\inf}}\mu(x\alpha s\gamma(p+q))\\
&=\underset{s\in S}{\underset{\alpha,\gamma\in\Gamma} {\inf}}\mu(x\alpha s\gamma p+x\alpha s\gamma q))\\
&\geq\underset{s\in S}{\underset{\alpha,\gamma\in\Gamma} {\inf}}\min\{\mu(x\alpha s\gamma p),\mu(x\alpha s\gamma q)\}\\
&=\min\{\underset{s\in S}{\underset{\alpha,\gamma\in\Gamma} {\inf}}\mu(x\alpha s\gamma p),\underset{s\in S}{\underset{\alpha,\gamma\in\Gamma} {\inf}}\mu(x\alpha s \gamma p)\}\\
&=\min\{<x,\mu>(p),<x,\mu>(q)\}
\end{array}
$$
Also,
$$
\begin{array}{ll}
<x,\mu>(p\beta q) &=\underset{s\in S}{\underset{\alpha,\gamma\in\Gamma} {\inf}} \mu(x\alpha s \gamma p\beta q)
\geq \underset{s\in S}{\underset{\alpha,\gamma\in\Gamma} {\inf}}\mu(x\alpha s\gamma p)=<x,\mu>(p)
\end{array}
$$
Now let p+a+z=b+z. So, $x\alpha s\gamma p+x\alpha s\gamma a+x\alpha s\gamma z=x\alpha s \gamma b+x\alpha s\gamma z.$  Then\\
$$
\begin{array}{ll}
<x,\mu>(p) &=\underset{s\in S}{\underset{\alpha,\gamma\in\Gamma} {\inf}}\mu(x\alpha s\gamma p)\\
&\geq\underset{s\in S}{\underset{\alpha,\gamma\in\Gamma} {\inf}}\min\{\mu(x\alpha s\gamma a),\mu(x\alpha s\gamma b)\}\\
&=\min\{\underset{s\in S}{\underset{\alpha,\gamma\in\Gamma} {\inf}}\mu(x\alpha s\gamma a),\underset{s\in S}{\underset{\alpha,\gamma\in\Gamma} {\inf}}\mu(x\alpha s\gamma b)\}\\
&=\min\{<x,\mu>(a),<x,\mu>(b)\}
\end{array}
$$
Hence $<x,\mu>$ is a fuzzy right h-ideal of S.
\end{proof}
\textbf{Note.} If $\mu$ is a fuzzy h-ideal of a commutative
$\Gamma$-hemiring S and $x \in S$, then the extension $<x, \mu>$ is
a fuzzy h-ideal of S.
\begin{Proposition} If $\mu_{i}, i=1,2,...$ be an arbitrary collection of fuzzy h-ideal of S, then
$<x,\underset{i}{\cap}\mu_{i}>$ is also a fuzzy h-ideal of S.
\end{Proposition}

\begin{Definition}\cite{Dutta} Let R, S be $\Gamma$-hemirings
and f: R$\rightarrow$ S be a function. Then f is said to be a
$\Gamma$-homomorphism if\\
(i) f(a+b)=f(a)+f(b),\\ (ii) f(a$\alpha b$)=f(a)$\alpha$f(b) for
a,b$\in$ R and $\alpha\in \Gamma$, \\(iii) $f(0_{R})=0_{S}$ where
$0_{R}$ and $0_{S}$  are the zeroes of R and S respectively.
\end{Definition}
\begin{Definition}\cite{Rosenfeld} Let f be a function from a
set X to a set Y; $\mu$ be a fuzzy subset of X and $\sigma$ be a
fuzzy subset
of Y.\\
Then image of $\mu$ under f, denoted by $f(\mu)$, is a fuzzy subset
of Y defined by\\
$f(\mu)(y)=\left \{\begin{array}{l} \underset{x\in f^{-1}{(y)}}{\sup
\mu(x)} ~~~\textmd{if}~~ f^{-1}(y)\neq \phi
\\ \textmd{0} ~~~\textmd{otherwise}
\end{array} \right . $\\
The pre-image of $\sigma$ under f, symbolized by $f^{-1}(\sigma),$
is a fuzzy subset of X defined by\\
$f^{-1}(\sigma)(x)=\sigma(f(x))$ $ \forall x \in$ X.
\end{Definition}
\begin{Proposition} Let f:R$\rightarrow$S be a morphism of
$\Gamma$-hemirings.\\
(i) If $\phi$ is a fuzzy right h-ideal of S, then $<z,f^{-1}(\phi)>$ is a
fuzzy right h-ideal of R, for any z$\in$R.\\
(ii) If f is surjective morphism and $\mu$ is a fuzzy right h-ideal
of R, then $<z,f(\mu)>$ is a fuzzy right h-ideal of S, for any z$\in$S.
\end{Proposition}
\begin{proof} Let f:R$\rightarrow$S be a morphism of
$\Gamma$-hemirings.\\
(i) Let $\phi$ be a fuzzy right h-ideal of S. Then by Proposition 17 of \cite{SD1}, we have $f^{-1}(\phi)$ is a fuzzy
right h-ideal of R. Now $<z,f^{-1}(\phi)>$ is an extension of $f^{-1}(\phi)$ in R. So, applying Theorem \ref{Prop 3.2} we
obtain that $<z,f^{-1}(\phi)>$ is a fuzzy right h-ideal of R.\\
(ii) Since f is surjective morphism and $\mu$ is a fuzzy right h-ideal
of R, by Proposition 17 of \cite{SD1}, we have $f(\mu)$ is a fuzzy right h-ideal of S. Hence with the help of Theorem
\ref{Prop 3.2} we get that $<z,f(\mu)>$ is a fuzzy right h-ideal of S, for any z$\in$S.
\end{proof}
\begin{Proposition} Let $\mu $ be a fuzzy h-ideal of S and $x \in S$.
Then the following conditions hold
\\(i) $\mu \subseteq <x,\mu>$,
\\(ii) $<(x \gamma)^{n-1}x,\mu> \subseteq <(x \gamma)^{n}x,\mu>$
where $\gamma \in \Gamma$,
\\(iii) If $\mu(x)>0$ then $supp~ <x,\mu>=S$ where $supp ~\mu$ is
defined by \\~~~~~~~$supp~ \mu=\{s \in S:\mu(s)>0\}$.\end{Proposition}

\begin{proof} (i) Let $y \in S$. Now
$<x,\mu>(y)=\underset{s\in S}{\underset{\alpha,\gamma\in\Gamma} {\inf}} \mu(x \alpha s\gamma  y) \geq \mu(y)$. Thus
$\mu \subseteq <x,\mu>$.
\\(ii) Let n be a positive integer and $y \in S$. Then
\\
$$
\begin{array}{ll}<(x \gamma)^{n}x,\mu>(y)&=\underset{s\in S}{\underset{\alpha,\beta\in\Gamma} {\inf}} \mu((x \gamma)^{n}x \alpha s\beta y)\\
& \geq
\underset{s\in S}{\underset{\alpha,\beta\in\Gamma} {\inf}} \mu((x \gamma )(x \gamma)^{n-1}x\alpha s \beta y) \\
& \geq \underset{s\in S}{\underset{\alpha,\beta\in\Gamma} {\inf}} \mu((x \gamma)^{n-1}x \alpha s\beta y)\\
&=<(x
\gamma)^{n-1}x,\mu> (y).
\end{array}
$$
So, $<(x \gamma)^{n-1}x,\mu> \subseteq
<(x \gamma)^{n}x,\mu>$.
\\(iii) Let $\mu(x)>0$ and $y \in S$. Then $<x,\mu>(y)=\underset{s\in S}{\underset{\alpha,\gamma\in\Gamma} {\inf}} \mu(x \alpha s\gamma y) \geq \mu(x)$.
Thus $y
\in supp <x,\mu>$ and consequently, $S \subseteq supp <x,\mu>$.
Hence $S =supp <x,\mu>$.
\end{proof}

\begin{Proposition} If $\mu$ is a fuzzy h-bi-ideal of S then its extension by x$\in$S, $<x,\mu>$ is also a fuzzy
h-bi-ideal of S, provided S is commutative.
\end{Proposition}
\begin{proof} Let $\mu$ is a fuzzy h-bi-ideal of S and its extension by x$\in$S is $<x,\mu>$. Since $\mu$ be a fuzzy h-bi-ideal it is sufficient to prove $<x,\mu>(p\alpha q \beta r)\geq\min\{<x,\mu>(p),<x,\mu>(r)\}$ for all p,q,r$\in$S and
$\alpha,\beta\in \Gamma$.\\
 Suppose p,q,r$\in$S and
$\alpha,\beta\in \Gamma$. Now\\
$<x,\mu>(p\alpha q \beta r)=\underset{s\in S}{\underset{\eta,\gamma\in\Gamma} {\inf}}\mu(x\eta s \gamma p\alpha q \beta r)\geq
\underset{s\in S}{\underset{\eta,\gamma\in\Gamma} {\inf}}\mu(x \eta s\gamma p)=<x,\mu>(p)$\\
Also, $<x,\mu>(p\alpha q \beta r)=\underset{s\in S}{\underset{\eta,\gamma\in\Gamma} {\inf}}\mu(x\eta s \gamma p\alpha q \beta r)\geq
\underset{s\in S}{\underset{\eta,\gamma\in\Gamma} {\inf}}\mu(x\eta s\gamma r)=<x,\mu>(r)$(since S is commutative).\\
Therefore $<x,\mu>(p\alpha q \beta r)\geq\min\{<x,\mu>(p),<x,\mu>(r)\}$. So, $<x,\mu>$ is a fuzzy h-bi-ideal of S.
\end{proof}
\begin{Proposition} If $\mu$ is a fuzzy h-interior-ideal of S then its extension by x$\in$S, $<x,\mu>$ is also a fuzzy
h-interior-ideal of S, provided S is commutative.
\end{Proposition}
\begin{proof} Let $\mu$ be a fuzzy h-interior-ideal of S and its extension by x$\in$S is $<x,\mu>$. Then it is sufficient to prove $<x,\mu>(p\alpha q \beta r)\geq <x,\mu>(q)$ for all p,q,r$\in$S and $\alpha,\beta\in \Gamma$.\\
Suppose p,q,r$\in$S and $\alpha,\beta\in \Gamma$. So,
$$
\begin{array}{ll}<x,\mu>(p\alpha q \beta r)&=\underset{s\in S}{\underset{\eta,\gamma\in\Gamma} {\inf}}\mu(x\eta s \gamma p\alpha q \beta r)=
\underset{s\in S}{\underset{\eta,\gamma\in\Gamma} {\inf}}\mu(x\eta s \gamma q\alpha p \beta r)\\
&\geq\underset{s\in S}{\underset{\eta,\gamma\in\Gamma} {\inf}}\mu(x\eta s \gamma q) =<x,\mu>(q)
\end{array}
$$
Hence $<x,\mu>$ is a fuzzy h-interior-ideal of S.
\end{proof}
\begin{Proposition}\label{ext-quasi} If $\mu$ is a fuzzy h-quasi-ideal of S then its extension by x$\in$S, $<x,\mu>$ is
also a fuzzy h-quasi-ideal of S.
\end{Proposition}
\begin{proof} Let $\mu$ be a fuzzy h-quasi ideal of S and its extension by x$\in$S is $<x,\mu>$. Suppose p, a,b,z$\in$S.
Then
$$
\begin{array}{ll}
<x,(\mu o_{h}\chi_{S})\cap(\chi_{S}o_{h}\mu)>(p)&=\underset{s\in S}{\underset{\alpha,\gamma\in\Gamma} {\inf}}((\mu
o_{h}\chi_{S})\cap(\chi_{S}o_{h}\mu))(x\alpha s \gamma p)\\
&=\underset{s\in S}{\underset{\alpha,\gamma\in\Gamma} {\inf}}\min\{(\mu o_{h}\chi_{S})(x\alpha s \gamma p),(\chi_{S}o_{h}\mu)(x\alpha s \gamma p) \}\\
&\leq \underset{~~~~~~~since~\mu~ is ~a~fuzzy~h-quasi-ideal}{\underset{s\in S}{\underset{\alpha,\gamma\in\Gamma} {\inf}}\min\{\mu(x\alpha s \gamma p),\mu(x\alpha s\gamma
p)\}}\\
&=\underset{s\in S}{\underset{\alpha,\gamma\in\Gamma} {\inf}}\mu(x\alpha s\gamma p)
=<x,\mu>(p)
\end{array}
$$
 Also from Theorem \ref{Prop 3.2} we have $<x,\mu>(p+q)\geq\min\{\mu(p),\mu(q)\}$ and p+a+z=q+z implies $<x,\mu>(p)\geq \min\{<x,\mu>(a),<x,\mu>(b)\}.$\\
Hence $<x,\mu>$ is a fuzzy h-quasi ideal of S.
\end{proof}
\begin{Remark} We know that if $\mu$ is fuzzy h-quasi-ideal of a $\Gamma$-hemiring S it is also a fuzzy h-bi-ideal. In
previous proposition \ref{ext-quasi} we show that its extension by any element x$\in$S, $<x,\mu>$ is a fuzzy
h-quasi-ideal also. Now it is a routine verification to show that $<x,\mu>$ is also a fuzzy h-bi-ideal of S, provided S is commutative.
\end{Remark}

\begin{Proposition} Let $\mu$ be a fuzzy h-ideal of S. Then for any x$\in$S, $<x,\mu^{+}>$ is also a fuzzy h-ideal of
S, where $\mu^{+}$ is defined by $\mu^{+}(x)=\mu(x)-\mu(0)+1$.
\end{Proposition}
\begin{proof} Since $\mu$ is a fuzzy h-ideal of S, by Proposition 25 of \cite{SD1} we have $\mu^{+}$ is also a fuzzy
h-ideal and hence by using Theorem \ref{Prop 3.2} we deduce that $<x,\mu^{+}>$ is also a fuzzy h-ideal of S.
\end{proof}
\begin{Proposition} If $\mu$ is a fuzzy h-ideal of S, then for any x$\in$ S, $<x,\mu_{\beta,\alpha}>$ is also a fuzzy
h-ideal of S, where $\mu_{\beta,\alpha}(y)=\beta.\mu(y)+\alpha$, $\beta\in(0,1]$ and $\alpha\in[0,1-\sup\{\mu(y): y\in
S \}]$.
\end{Proposition}
\begin{proof} Since $\mu$ is a fuzzy h-ideal of S, by Theorem 20 of \cite{SD1} we have $\mu_{\beta,\alpha}$ is also a
fuzzy h-ideal and hence by using Theorem \ref{Prop 3.2} we deduce that $<x,\mu_{\beta,\alpha}>$ is also a fuzzy h-ideal
of S.
\end{proof}
\begin{Proposition} If $\mu$ and $\nu$ are any two fuzzy h-ideal of S, then for any x$\in$ S, $<x,\mu\times\nu>$ is also
a fuzzy h-ideal of S, where $(\mu\times\nu)(a,b)=\min\{\mu(a),\mu(b)\},~ a,b\in S$.
\end{Proposition}
\begin{proof} Since $\mu$ and $\nu$ be any two fuzzy h-ideal of S, by Theorem 35 of \cite{SD1} we have $\mu\times\nu$
is also a fuzzy h-ideal and hence by using Theorem \ref{Prop 3.2} we deduce that $<x,\mu\times\nu>$ is also a fuzzy
h-ideal of S.
\end{proof}
\begin{Theorem} Let $\mu$, $\nu$ be any two fuzzy h-ideal of S and x,y$\in$S. Then $<x,\mu>\times<y,\nu>$ is also a
fuzzy h-ideal of S.
\end{Theorem}
\begin{proof} Since $\mu$ and $\nu$ be any two fuzzy h-ideal of S, by Theorem \ref{Prop 3.2} we have $<x,\mu>$ and
$<y,\nu>$ are fuzzy h-ideals of S. Hence by using Theorem 35 of \cite{SD1} we deduce that $<x,\mu>\times<y,\nu>$ is
also a fuzzy h-ideal of S.
\end{proof}

\begin{Proposition} Let $\mu $ be a prime fuzzy h-ideal of S. Then for
all x,y $\in$ S, $$\underset{s\in S}{\underset{\alpha,\gamma\in\Gamma} {\inf}}\mu(x\alpha s \gamma  y)=\max [\mu(x),\mu(y)].$$
\\Conversely, let $\mu $ be a fuzzy h-ideal of S such that $Im ~\mu=\{1,t\}, t \in
[0,1)$.\\
If $\underset{s\in S}{\underset{\alpha,\gamma\in\Gamma} {\inf}} \mu(x\alpha s \gamma  y)=\max [\mu(x),\mu(y)]$
for all x,y $\in$ S then $\mu $ is a prime fuzzy h-ideal of
S.\label{Th:3.4}\end{Proposition}

\begin{proof} Let $\mu $ be a prime fuzzy h-ideal of S. Then by
Theorem \ref{Charc. Prime} we have, $$\underset{s\in S}{\underset{\alpha,\gamma\in\Gamma} {\inf}}\mu(x\alpha s \gamma y)=1 ~ or~ t.$$
\\Case I. Let $\max [\mu(x),\mu(y)]=1$. Then suppose that
$\mu(x)=1$. Consequently, $x \in \mu_{0}$. As $\mu_{0}$ is an h-ideal
of S, $x\alpha s \gamma y \in \mu_{0}$ for all s $\in$ S and $ \gamma \in \Gamma$. Thus
$$\underset{s\in S}{\underset{\alpha,\gamma\in\Gamma} {\inf}} \mu(x\alpha s \gamma y)=1=\max [\mu(x),\mu(y)].
$$
\\Case II. Let $\max [\mu(x),\mu(y)]=t$. Then
$\mu(x)=\mu(y)=t$. This implies that $x,y \not\in \mu_{0}$.
Since $\mu_{0}$ is a prime h-ideal of S, so $x \Gamma S \Gamma y \not
\subseteq \mu_{0}$. Thus there exists some $\alpha_{1},\gamma_{1}\in \Gamma$ and $s_{1}\in S$ such that $x \alpha_{1} s_{1}\gamma_{1} y \not\in \mu_{0}$, i.e.,
$\mu(x\alpha_{1} s_{1}\gamma_{1} y) \neq 1$. Therefore $\mu(x \alpha_{1}s_{1}\gamma_{1}y) = t$. Thus $$\underset{s\in S}{\underset{\alpha,\gamma\in\Gamma} {\inf}}\mu(x\alpha s
\gamma y)= t=\max
[\mu(x),\mu(y)]. $$
\\For the converse part, let $x,y \in S$ such that $x \Gamma S \Gamma y
\subseteq \mu_{0}$. Then $x \alpha s \gamma y \in \mu_{0}$ for all $ \alpha,\gamma\in \Gamma$ and $s\in S$. So
$\underset{s\in S}{\underset{\alpha,\gamma\in\Gamma} {\inf}} \mu(x\alpha s \gamma y)=1=\max [\mu(x),\mu(y)]
$. This implies that either $\mu(x)=1$ or $\mu(y)=1$, i.e., either
$x \in \mu_{0}$ or $y \in \mu_{0}$. Consequently, $\mu_{0}$ is a
prime h-ideal of S by Theorem \ref{Th:2.3}. Hence by Theorem
\ref{Charc. Prime}, $\mu$ is a prime fuzzy h-ideal of S.
\end{proof}

\begin{Proposition} If $\mu$ is a prime(semiprime) fuzzy h-ideal of S then its extension by x$\in$S, $<x,\mu>$ is also
a prime(semiprime) fuzzy h-ideal of S.
\end{Proposition}
\begin{proof} Let $\mu$ be a prime fuzzy h-ideal of S and its extension by x$\in$S is $<x,\mu>$. Let p,q,$s_{1}\in$ S and $\delta,\eta\in \Gamma$. Then\\
$$
\begin{array}{ll}\underset{s_{1}\in S}{\underset{\delta,\eta\in\Gamma} {\inf}} <x,\mu>(p\delta s_{1} \eta q)&=\underset{s_{1}\in S}{\underset{\delta,\eta\in\Gamma} {\inf}}\underset{s\in S}{\underset{\alpha,\gamma\in\Gamma} {\inf}}\mu(x\alpha s \gamma p\delta s_{1} \eta  q)\\
&=\underset{s\in S}{\underset{\alpha,\gamma\in\Gamma} {\inf}}\{max(\mu(x\alpha s\gamma p),\mu(q))\}\\
&=\max\{ \underset{s\in S}{\underset{\alpha,\gamma\in\Gamma} {\inf}}\mu(x\alpha s\gamma p), \mu(q)\}\\
&=\max\{\mu(x),\mu(p),\mu(q)\}\\
&=\max\{ \max\{\mu(x),\mu(p)\},\max\{\mu(x),\mu(q)\} \}\\
&=\max\{\underset{s\in S}{\underset{\alpha,\gamma\in\Gamma} {\inf}}\mu(x\alpha s\gamma p),\underset{s\in S}{\underset{\alpha,\gamma\in\Gamma} {\inf}}\mu(x\alpha s\gamma q)\}\\
&=\max\{ <x,\mu>(p),<x,\mu>(q)\}
\end{array}
$$

Similarly, we can prove the result for semiprime fuzzy h-ideal.

\end{proof}

\begin{Proposition} Let $\mu $ be a prime fuzzy h-ideal of S and $x \in
S$. Then \\$<x,\mu>(y)=\displaystyle{\inf_{\begin{array}{l} s_{1}\in S
\\\eta,\delta \in
\Gamma \end{array}}} [<x \eta s_{1} \delta x,\mu> (y)]$ for all $y \in
S$.\end{Proposition}

\begin{proof} Let $y \in S$. Then
$$
\begin{array}{ll}\displaystyle{\inf_{\begin{array}{l} s_{1}\in S
\\ \eta,\delta \in
\Gamma \end{array}}} [<x \eta s_{1} \delta x,\mu>
(y)]&=\displaystyle{\inf_{\begin{array}{l} s_{1}\in S
\\\eta,\delta \in
\Gamma \end{array}}}\underset{s\in S}{\underset{\alpha,\gamma\in\Gamma} {\inf}}\mu(x \eta s_{1} \delta x\alpha s\gamma y)
\\&=\displaystyle{\inf_{\begin{array}{l} s_{1}\in S
\\ \eta,\delta \in
\Gamma \end{array}}} \max [\mu(x\eta s_{1} \delta
x),\mu(y)]\\
&=\max[\displaystyle{\inf_{\begin{array}{l} s_{1}\in S
\\ \eta,\delta \in
\Gamma \end{array}}} \mu(x \eta s_{1} \delta
x),\mu(y)]\\
&=\max[\max[\mu(x),\mu(x)],\mu(y)]=\max[\mu(x),\mu(y)]\\
&=\underset{s\in S}{\underset{\alpha,\gamma\in\Gamma} {\inf}}\mu(x\alpha s \gamma y)=<x,\mu>(y)
\end{array}
$$
\end{proof}

\begin{Definition} Let $A \subseteq S$ and $x \in S$. Then we define
$<x,A>$ as \\$<x,A>=\{ y \in S: x \Gamma S \Gamma y \subseteq A\}$.
\end{Definition}

\begin{Proposition} Let A be a non empty subset of S. Then
$<x,\lambda_{A}>=\lambda_{<x,A>}$ for every $x \in S$ where
$\lambda_{A}$ denotes the characteristic function of
A.\end{Proposition}

\begin{proof} Let $y \in S$. Then
$<x,\lambda_{A}>(y)=\underset{s\in S}{\underset{\alpha,\gamma\in\Gamma} {\inf}} \lambda_{A}(x\alpha s \gamma y)=1$ or $0$.
\\Case I. If $<x,\lambda_{A}>(y)=1$ then $\lambda_{A}(x\alpha s \gamma y)=1$ for all $ \alpha,\gamma \in \Gamma$ and $s\in S$.
Thus $x \alpha s \gamma y \in A$ for all
$\alpha,\gamma\in \Gamma$ and $s\in S$, i.e., $x \Gamma S \Gamma y \subseteq A$.
i.e., $y \in <x,A>$. Consequently $\lambda_{<x,A>}(y)=1$. So,
$<x,\lambda_{A}>=\lambda_{<x,A>}$.
\\Case II. If $<x,\lambda_{A}>(y)=0$ then $\lambda_{A}(x\alpha s \gamma y)=0$ for some $ \alpha,\gamma \in \Gamma$ and $s\in S$. So $x\alpha s
\gamma y \not\in A$ for some $ \alpha,\gamma\in \Gamma$ and $s\in S$. Therefore $x \Gamma S \Gamma y \not \subseteq A$ which
implies that $y \not\in <x,A>$. Thus $\lambda_{<x,A>}(y)=0$. Hence
$<x,\lambda_{A}>=\lambda_{<x,A>}$. This proves the proposition.
\end{proof}

\begin{Theorem} Let $\mu $ be a prime fuzzy h-ideal of S and $x \in
S$ be such that $x \not\in \mu_{0}$. Then $<x,\mu>=\mu$.
\\Conversely, let $\mu $ be a fuzzy h-ideal of S such that $Im~ \mu=\{1,t\}, t \in
[0,1)$. If $<x,\mu>=\mu$ for those $x \in S$ for which
$\mu(x)=\alpha$, then $\mu$ is a prime fuzzy h-ideal of
S.\label{Th:3.8}\end{Theorem}

\begin{proof} Let $\mu $ be a fuzzy prime h-ideal of S. Then by
Theorem \ref{Charc. Prime}, \\(i) $\mu(0)=1$, (ii) $ Im ~\mu =\{1,
t\}$ with $t \in [0,1)$,
\\(iii) $\mu_{0}=\{x \in S: \mu(x)=\mu(0)\}$ is a prime h-ideal of
S.
\\Case I. If $ y \in \mu_{0}$ then $x \alpha s\gamma y \in \mu_{0}$ for all $ \alpha,\gamma \in\Gamma$ and $s\in S$. So
$\underset{s\in S}{\underset{\alpha,\gamma\in\Gamma} {\inf}} \mu(x\alpha s \gamma y)=1=\mu(y)$. Therefore
$<x,\mu>=\mu$.
\\Case II. Let $y \not\in \mu_{0}$. Since $\mu_{0}$ is a prime h-ideal
of S and $x,y \not\in \mu_{0}$, so $x \Gamma S \Gamma y \not
\subseteq \mu_{0}$. Therefore there exist some
$\alpha,\gamma \in \Gamma$ and $s\in S$ such that $x\alpha s \gamma y \not\in\mu_{0}$. Therefore $\mu(x\alpha s \gamma  y)=t$. Thus
$\underset{s\in S}{\underset{\alpha,\gamma\in\Gamma} {\inf}}\mu(x\alpha s \gamma y)=t=\mu(y)$ and
hence $<x,\mu>=\mu$. Consequently, $<x,\mu>=\mu$ for all $x \in
\mu_{0}$. \\Conversely, let $x,y \in S $.
\\Case I. Let $\mu(x)=t$. Then $\max \{ \mu(x), \mu(y)
\}=\mu(y)$. Now $\mu(y)=<x,\mu>(y)$ implies that $\max \{ \mu(x),
\mu(y) \}=\underset{s\in S}{\underset{\alpha,\gamma\in\Gamma} {\inf}}\mu(x\alpha s \gamma  y)$.
\\Case II. Let $\mu(x)=1$. Then $x \in \mu_{0}$. So $x \alpha s \gamma  y \in \mu_{0}$ for
all $\alpha, \gamma \in \Gamma$ and $s\in S$. Thus
$\underset{s\in S}{\underset{\alpha,\gamma\in\Gamma} {\inf}} \mu(x\alpha s \gamma  y)=1=\mu(x)=\max \{
\mu(x), \mu(y) \}$.\\Thus $\underset{s\in S}{\underset{\alpha,\gamma\in\Gamma} {\inf}}\mu(x\alpha s \gamma y)=\max \{ \mu(x), \mu(y)
\}$ for all $x,y \in S$.
\\Hence by the converse part of Theorem \ref{Th:3.4}, $\mu$ is a
prime h-ideal of S.
\end{proof}

\begin{Theorem} Let $\mu$ be a prime fuzzy h-ideal of S and $x \in S$
be such that $x \in \mu_{0}$. Then
$<x,\mu>=\textbf{1}_{S}$.\label{Th:3.9}\end{Theorem}

\begin{proof} Let $\mu $ be a prime fuzzy h-ideal of S and y $\in$ S. Then $x \gamma y \in \mu_{0}$ for all $s
\in S$ and for all $ \gamma\in \Gamma$ as $x \in \mu_{0}$.
\\So $<x,\mu>(y)=\underset{s\in S}{\underset{\alpha,\gamma\in\Gamma} {\inf}} \mu(x \gamma y)=1=\textbf{1}_{S}(y)$
for all $y \in S$. Hence $<x,\mu>=\textbf{1}_{S}$.
\end{proof}

\begin{Corollary} Let I be an h-ideal of S. If I is prime h-ideal of S
then for $x \in S$, $<x,\lambda_{I}>=\lambda_{I}$ where $x \not\in
I$.\end{Corollary}

\begin{proof} Let I be a prime h-ideal of S. Then $\lambda_{I}$ is a
prime fuzzy h-ideal of S. Now $x \not\in I$ implies that $x \not\in
(\lambda_{I})_{0}$. Hence by Theorem \ref{Th:3.8},
$<x,\lambda_{I}>=\lambda_{I}$.
\end{proof}
%Combining Theorem \ref{Th:3.8} and Theorem \ref{Th:3.9} we obtain
%the following Theorem.

%\begin{Theorem} Let $\mu$ be a fuzzy prime h-ideal of S and $x \in
%S$. Then either $<x,\mu>$ is a prime h-ideal of S or $<x,\mu>$ is
%constant.\end{Theorem}

\begin{Theorem} Let S be a commutative $\Gamma$-hemiring and $\mu$
be a fuzzy subset of S such that $<x,\mu>=\mu$ for all $x \in S$.
Then $\mu$ is constant.\end{Theorem}

\begin{proof} Let $x,y \in S$. Then
$\mu(y)=<x,\mu>(y)=\underset{s\in S}{\underset{\alpha,\gamma\in\Gamma} {\inf}} \mu(x\alpha s \gamma
y)=\underset{s\in S}{\underset{\alpha,\gamma\in\Gamma} {\inf}}\mu(y\alpha s \gamma x)$, (since S is a
commutative $\Gamma$-hemiring)$=<y,\mu>(x)=\mu(x)$.
\\Therefore $\mu(x)=\mu(y)$ for all $x,y \in S$. Hence $\mu$ is
constant.
\end{proof}

\end{document}